%% file: pseudo20-without-line-numbers.tex
\documentclass[12pt,letterpaper]{article}

\usepackage[utf8x]{inputenx}
\usepackage[T1]{fontenc}
\usepackage{textcomp}
\usepackage{lmodern}

\usepackage{soul}
\usepackage{verbatim}

\usepackage{amsthm, amsmath, amssymb}
\usepackage{amsfonts}
\usepackage{mathcomp}
\usepackage{mathrsfs}
\usepackage{euscript}

\usepackage{graphics}
\usepackage{graphicx}
\usepackage{caption}
\usepackage{booktabs}
\usepackage{floatrow} 
\usepackage{subfig}
\usepackage{psfrag}
\usepackage{epic,eepic}
\usepackage{color}
\usepackage[table]{xcolor}
\usepackage{tikz}

\usepackage{cite}
\usepackage{enumerate}
\usepackage{paralist}
\usepackage{amsrefs}   

\usepackage{fullpage}

\theoremstyle{plain}
\newtheorem{thm}{Theorem}
\newtheorem{lem}[thm]{Lemma}

\newtheorem{prop}[thm]{Proposition}

\theoremstyle{definition}

\newtheorem{conj}[thm]{Conjecture}

\theoremstyle{remark}

\def\ga{{G(\aa)}}
\def\ga{{G_\aa}}
\def\gam{{G(\aa),w_m}}
\def\gar{{G(\rr),w_m}}
\def\gam{{G_\aa,w_m}}
\def\gar{{G_\rr,w_m}}
\def\dd{{\mathcal D}}

\def\aa{{\mathcal A}}
\def\rr{{\mathcal R}}

\def\real{{\mathbb{R}}}
\def\natu{{\mathbb{N}}}

\def\sphe{{\mathbb{S}}}
\def\lar{{\leftrightarrow}}
\def\ucr#1{{\operatorname{cr}(#1)}}
\def\rcr#1{{\overline{\operatorname{cr}}(#1)}}
\def\pcr#1{{\widetilde{\operatorname{cr}}(#1)}}
\def\tucr{{\operatorname{cr}}}
\def\trcr{{\overline{\operatorname{cr}}}}
\def\tpcr{{\widetilde{\operatorname{cr}}}}
\def\moncr#1{{{\operatorname{mon-cr}}(#1)}}

\def\pp{{\mathcal P}}

\def\ignore#1{{}}

\newcommand*{\bqed}{\hfill\ensuremath{\blacksquare}}

\title{\bf On the pseudolinear crossing number}

\author{C\'esar Hern\'andez-V\'elez\thanks{Instituto de Matem\'atica e Estat\'{\i}stica, Universidade de S\~ao Paulo. S\~ao Paulo, Brasil 05508-090. {\em E-mail:} {\tt israel@ime.usp.br}. Supported by FAPESP (Proc. 2012/24597-3).}
\and
Jes\'us Lea\~nos \thanks{Unidad Acad\'emica de Matem\'aticas, Universidad Aut\'onoma de Zacatecas. Zacatecas, M\'exico, 98000. {\em E-mail}: {\tt jleanos@matematicas.reduaz.mx}. Supported by CONACyT Grant 179867.}
\and
Gelasio Salazar\thanks{Instituto de F\'{\i}sica,
Universidad Aut\'onoma de San Luis Potos\'{\i}. 
San Luis Potos\'{\i}, M\'exico. {\em E-mail:} {\tt gsalazar@ifisica.uaslp.mx}. Supported by CONACyT Grant 106432.
}}

\begin{document}


\maketitle

\begin{abstract}
  A drawing of a graph is {\em pseudolinear} if there is a pseudoline arrangement such that each pseudoline contains exactly one edge of the drawing. The {\em pseudolinear crossing number} $\pcr{G}$ of a graph $G$ is the minimum number of pairwise crossings of edges in a pseudolinear drawing of $G$. We establish several facts on the pseudolinear crossing number, including its computational complexity and its relationship to the usual crossing number and to the rectilinear crossing number. This investigation was motivated by open questions and issues raised by Marcus Schaefer in his comprehensive survey of the many variants of the crossing number of a graph.



\end{abstract} 


\section{Introduction}
\label{sec:intro}

In his comprehensive survey of the many variants of the crossing number of a graph, Schaefer~\cite{survey} brought up several issues regarding the pseudolinear crossing number, including its computational complexity and its relationship to other variants of crossing number. Our aim in this paper is to settle some of these issues.

A {\em pseudoline} is a simple closed curve in the projective plane $\mathbb{P}^2$ which does not disconnect $\mathbb{P}^2$. A {\em pseudoline arrangement} is a set of pseudolines that pairwise intersect (necessarily, cross) each other exactly once.

Let $\dd$ be a drawing of a graph $G$ in the plane, and let $C$ be a disk containing $\dd$. By identifying antipodal points on the boundary of $C$ and discarding $\mathbb{R}^2 \setminus C$ we may regard $\dd$ as lying in $\mathbb{P}^2$. If each edge can be extended to a pseudoline so that the result is a pseudoline a\-rran\-ge\-ment, then $\dd$ is a {\em pseudolinear drawing}. The {\em pseudolinear crossing number $\pcr{G}$} of $G$ is the minimum number of pairwise crossings of edges in a pseudolinear drawing of $G$.

We recall that the {\em crossing number $\ucr{G}$} of a graph $G$ is the minimum number of pairwise crossings of edges in a drawing of $G$ in the plane. A drawing in which each edge is a straight line segment is a {\em rectilinear} drawing. The {\em rectilinear crossing number $\rcr{G}$} of $G$ is the minimum number of pairwise crossings of edges in a rectilinear drawing of $G$. A rectilinear drawing is clearly pseudolinear. Since pseudolinear and rectilinear drawings are restricted classes of drawings, it follows that for any graph $G$ we have $\ucr{G} \le \pcr{G}\leq \rcr{G}$.

The decision problem {\sc CrossingNumber}, which takes as input a graph $G$ and an integer $k$, and asks if $\ucr{G}\le k$, is NP-complete~\cite{GaJo83}. It is not difficult to prove that {\sc RectilinearCrossingNumber} (the corresponding variant for $\rcr{G}$) is NP-hard (cf.~Lemma~\ref{lem:pcrisNPhard} below). Bienstock's reduction from {\sc Stretchability} to {\sc RectilinearCrossingNumber}~\cite{Bi91} implies that computing the rectilinear crossing number is $\exists\real$-complete (see Section~\ref{sub:proofs}).

In~\cite{survey}, Schaefer listed the complexity of {\sc PseudolinearCrossingNumber} (the corresponding variant for $\pcr{G})$ as an open problem. Here we settle this question as follows.

\begin{thm} \label{thm:np-complete}
{\sc PseudolinearCrossingNumber} is NP-complete.
\end{thm}

Bienstock and Dean~\cite{BiDe93} showed that for any integers $k,m$ with $m\ge k \ge 4$, there is a graph $G$ with $\ucr{G}=k$ and $\rcr{G}\ge m$. In~\cite{survey}, Schaefer wrote: ``Bienstock and Dean's graphs $G_m$ with $\ucr{G_m}=4$ and $\rcr{G_m}=m$ should give $\pcr{G_m}=\rcr{G_m}$, since the proof of $\rcr{G_m}\ge m$ seems to work with pseudolinear drawings.'' As we set to work out the details, we realized that the Bienstock and Dean proof does not carry over to the pseudolinear case in a totally straightforward way: an obstacle to extend a set of segments to an arrangement of pseudolines needs to be found. As it is often the case when settling a stronger result, our proof of the following statement turned out to be simpler than the proof in~\cite{BiDe93}. For this reason, and because this also implies the Bienstock and Dean result, it seems worth to include here the following statement and its proof.

\begin{thm}\label{thm:fourandmore}
For any integers $k, m$ with $m\ge k \ge 4$, there is a graph $G$ with $\ucr{G}=k$ and $\pcr{G} \ge m$.
\end{thm}

As Schaefer observes, this also separates the monotone crossing number mon-cr from the pseudolinear crossing number, since for any graph $G$ we have mon-cr$(G) \le \binom{2\text{\rm cr}(G)}{2}$~\cite{pachtoth}.

Although pseudoline arrangements are defined in $\mathbb{P}^2$, we can alternatively think of them as lying in the Euclidean plane $\mathbb{R}^2$: starting with the $\mathbb{P}^2$ representation, we delete the disk boundary and extend infinitely (to rays) the segments that used to intersect the disk boundary. An arrangement of pseudolines may then be naturally regarded as a cell complex covering the plane. Two arrangements are {\em isomorphic} if there is a one-to-one adjacency-preserving correspondence between the objects in their associated cell complexes. Ringel~\cite{Ri55} was the first to exhibit a pseudoline arrangement (in $\real^2$) that is {\em non-stretchable}, that is, not isomorphic to any arrangement in which every pseudoline is a straight line. 

Schaefer wrote in~\cite{survey}: ``It should be possible to take a non-stretchable pseudoline arrangement $A$ and use Bienstock's machinery~\cite{Bi91} to build a graph $G_A$ for which $\pcr{G_A} < \rcr{G_A}$.'' Using Schaefer's roadmap, we have constructed a family of graphs to prove the following.

\begin{thm}\label{thm:pseu-rec}
For each integer $m \ge 1$ there exists a graph $G$ such that $\pcr{G} = 36(1+4m)$ and $\rcr{G} \ge 36(1+4m)+m$.
\end{thm}

Yet another reason that makes worth to include in its full detail the construction proving this last result, is that we use it to prove the following.

\begin{thm}\label{thm:same}
The decision problem ``Is $\pcr{G}=\rcr{G}$''? is $\exists\real$-complete.
\end{thm}

Theorems~\ref{thm:np-complete} and \ref{thm:fourandmore} are proved in Sections~\ref{sec:np-complete} and \ref{sec:fourandmore}, respectively. Theorems~\ref{thm:pseu-rec} and \ref{thm:same} are proved in Section~\ref{sec:pseu-rec}. Section~\ref{sec:cr} contains some concluding remarks and open questions.

\subsection{Observations and terminology for the rest of the paper}

Unless otherwise stated, a drawing is understood to be a drawing in $\real^2$.  All drawings of a graph $G$ under consideration either minimize $\ucr{G}$, or are pseudolinear or rectilinear drawings of $G$. All such drawings are {\em good}, that is, no two edges cross each other more than once, no adjacent edges cross each other, and no edge crosses itself. Thus we implicitly assume that all drawings under consideration are good. A drawing $\dd$ (in any surface $\Sigma$) may be regarded as a one-dimensional subset of $\Sigma$. Taking this viewpoint, a {\em region} of $\dd$ is a connected component of $\Sigma\setminus \dd$. Thus, in the particular case in which $\dd$ is an embedding, the regions of $\dd$ are simply the faces. Finally, two drawings $\dd$ and $\dd'$ of the same graph in a surface $\Sigma$ are {\em isomorphic} if there is a self-homeomorphism of $\Sigma$ that takes $\dd$ to $\dd'$.


\section{Complexity of {\sc PseudolinearCrossingNumber}: \\ proof of Theorem~\ref{thm:np-complete}}\label{sec:np-complete}

We prove NP-hardness in Lemma~\ref{lem:pcrisNPhard} and membership in NP in Lemma~\ref{lem:pcrisinNP}.

The fact that {\sc PseudolinearCrossingNumber} is NP-hard is not difficult to prove, and although we could not find any reference in the literature, perhaps it could be considered a folklore result. It seems worth to include this proof, for completeness.

\begin{lem}\label{lem:pcrisNPhard}
{\sc PseudolinearCrossingNumber}, {\sc RectilinearCrossingNumber}, and {\sc Mo\-notoneCrossingNumber} are NP-hard.
\end{lem}

\begin{proof}
We claim that for any graph $G$  there is a graph $G'$ obtained by subdividing each edge of $G$ at most $2|E(G)|$ times, and such that $\rcr{G'}=\ucr{G}$. We note that the {\sc Rectilinear\-CrossingNumber} part of the lemma follows at once from this claim. The other statements also follow, since $\ucr{G} \le \moncr{G} \le \pcr{G} \le \rcr{G}$ hold for any graph $G$.

We now prove the claim. Let $\dd$ be a crossing-minimal drawing of $G$. A {\em segment} of $\dd$ is an arc of $\dd$ whose endpoints are either two vertices, or one vertex and one crossing, or two crossings, and is minimal with respect to this property. (Put differently, if we planarize $\dd$ by converting crossings into degree $4$ vertices, the segments correspond to the edges of this plane graph). By F\'ary's theorem~\cite{fary}, every planar graph has a plane rectilinear drawing. Therefore there is a drawing $\dd'$ of $G$, with the same number of crossings as $\dd$, in which every segment is straight. Now for each edge $e$ of $G$, let $\times(e)$ denote the number of crossings of $e$. It is easy to see that if we subdivide each edge $e$ a total of $2\cdot\times(e)$ times, then the resulting graph $G'$ has a rectilinear drawing with $\ucr{G}$ crossings: indeed, it suffices to place two pairs of new (subdivision) vertices in a small neighborhood of each crossing of $\dd'$, one pair on each of the crossing edges, and join each pair with a straight segment.
\end{proof}

\def\nice{{pseudolinear model}}

We now settle membership in NP. A {\em {\nice}} graph is a plane graph $H$ with two disjoint distinguished subsets of vertices $T=\{t_1,t_2,\ldots,t_{2m}\}$ (where each {\em terminal} $t_i$ has degree $1$) and $V$, such that the following hold:

\begin{enumerate}
\item The boundary walk (say, in clockwise order) along the infinite face has the vertices $t_1, t_2, \ldots, t_{2m}$ (but not necessarily only these vertices) in this cyclic order.
\item There is a collection of paths $\pp=\{P_1, P_2, \ldots,P_m\}$ in $H$ with the following properties:
\begin{enumerate}
\item $H=P_1\cup P_2 \cup \cdots \cup P_m$. 
\item The ends of $P_i$ are $t_i$ and $t_{i+m}$, for $i=1,2,\ldots,m$.
\item Each $P_i$ contains exactly two vertices in $V$.
\item Any two paths in $\pp$ intersect each other in exactly one vertex, and if they intersect in a vertex not in $V$, then this vertex has degree $4$.
\end{enumerate}

\end{enumerate}

For each $i=1,2,\ldots,m$, let $u_i, v_i$ be the (only two) vertices in $V$ contained in $P_i$. Then the interior vertices of the subpath $u_i P_i v_i$ (if any) are {\em special} vertices of $H$. This {\nice} $H$ {\em induces} a graph $G$ with vertex set $V$, where $u,v\in V$ are adjacent in $G$ if and only if there is a path in $\pp$ that contains $u$ and $v$. 

\begin{lem}\label{lem:pcrisinNP}
{\sc PseudolinearCrossingNumber} is in NP.
\end{lem}

\begin{proof}
The key claim is that a graph $G=(V,E)$ has a pseudolinear drawing with exactly $k$ crossings if and only if $G$ is induced by a {\nice} with exactly $k$ special vertices.

For the ``only if'' part, suppose that $G$ has a pseudolinear drawing with $k$ crossings. Extend the edges of $G$ so that the resulting pseudolines form an arrangement; this can clearly be done so that no more than two pseudolines intersect at a given point, unless this point is in $V$. By transforming the edge crossings to (degree $4$, special) vertices, and transforming into vertices the intersections of the pseudolines with the disk boundary, the result is a {\nice} plane graph $H$ with exactly $k$ special vertices. For the ``if'' part, suppose that $G$ is induced by a {\nice} graph with $k$ special vertices. Consider then the drawing of $G$ obtained by removing all vertices that are neither in $V$ nor special, and then  transforming  each special vertex into a crossing. The result is a pseudolinear drawing of $G$ with $k$ crossings.

Thus the existence of a {\nice} graph $H$ with $k$ special vertices that induces $G$ provides a certificate that the pseudolinear crossing number of $G$ is at most $k$. Since the size of $H$ is clearly polinomially bounded on the size of $G$, the lemma follows.
\end{proof}


\section{Separating $\tpcr$ from $\tucr$: proof of Theorem~\ref{thm:fourandmore}}\label{sec:fourandmore}

We start by finding a substructure that guarantees that a drawing is not pseudolinear. 
A {\em clam} is a drawing of two disjoint $2$-paths $P$ and $Q$, with exactly two faces in which the infinite face is incident with the internal vertices of $P$ and $Q$, and with no other vertices.
It is easy to see that, up to isomorphism, a clam drawing looks as the one depicted in Figure~\ref{fig:forb}.  

\begin{center}
\begin{figure}[ht!]
\scalebox{0.5}{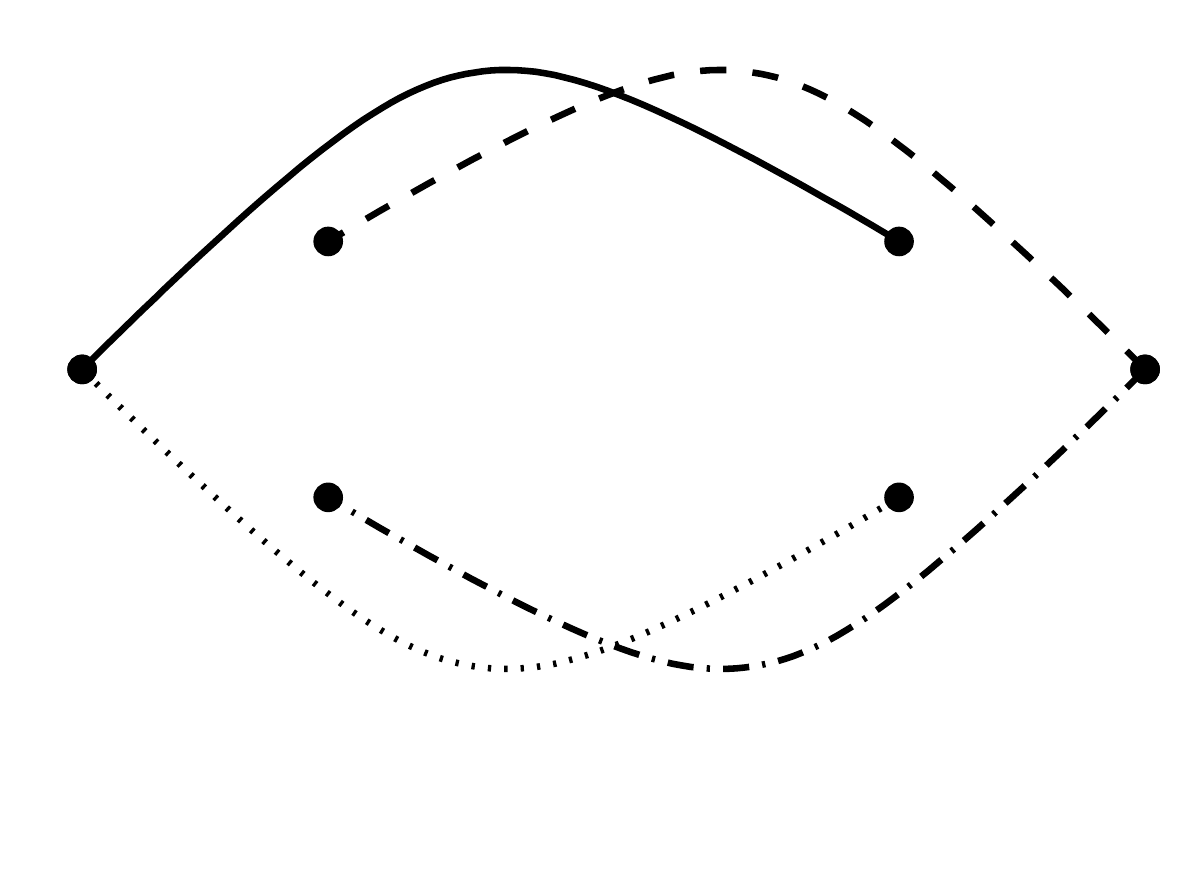}
\caption{A clam.}
\label{fig:forb}
\end{figure}
\end{center}

\begin{prop}[An obstacle to pseudolinearity]\label{pro:clam}
Let $P,Q$ be disjoint $2$-paths of a graph $G$. If $\dd$ is a drawing of $G$ whose restriction to $P\cup Q$ is a clam, then $\dd$  is not pseudolinear.
\end{prop}

\begin{proof}
It clearly suffices to show that the restriction $\dd'$ of $\dd$ to $P\cup Q$ is not pseudolinear. Without any loss of generality we may assume that $\dd'$ is as shown in Figure~\ref{fig:forb}.

By way of contradiction, suppose that $\dd'$ is pseudolinear. Thus there exists a disc $C$ that contains $\dd'$, such that in the projective plane that results by identifying antipodal points of $C$, there is a pseudoline arrangement $\{\ell_1,\ell_2,\ell_3,\ell_4\}$ where $\ell_i$ contains $e_i$ for $i=1,2,3,4$. Since $s$ is not incident with the infinite region of $\dd'$, it follows that $\ell_1$ must intersect the boundary of the infinite region at some point in $e_4$ between $v$ and $y$ (if the intersection occurred elsewhere, $\ell_1$ would intersect another pseudoline more than once). Totally analogous arguments show that $\ell_2$ intersects $e_3$ at some point between $v$ and $x$; $\ell_3$ intersects $e_2$ at some point between $u$ and $y$; and $\ell_4$ intersects $e_1$ at some point between $u$ and $x$. Together with $u, x, v$, and $y$, this gives $8$ intersections between the $4$ pseudolines, contradicting that any pseudoline arrangement with $4$ pseudolines has $\binom{4}{2}=6$ intersection points.
\end{proof}

\begin{proof}[Proof of Theorem~\ref{thm:fourandmore}]
Consider the graph $G$ drawn in Figure~\ref{fig:diam}. The edges drawn as thick, continuous segments are {\em heavy}. The other edges (the dotted ones) are {\em light}. 
We regard the drawing $\dd$ of $G$ in Figure~\ref{fig:diam} as a drawing in the sphere $\sphe^2$. We say that a drawing of $G$ (in either $\sphe^2$ or $\real^2$) is {\em clean} if no heavy edge is crossed. 

\bigskip
\noindent{\bf Claim.} {\em No clean drawing of $G$ in $\real^2$ is pseudolinear.}\\

\noindent{\em Proof.}
Up to isomorphism, there are exactly two clean drawings of $G$ in $\sphe^2$, which correspond to the two different embeddings of the subgraph of $G$ induced by the heavy edges. One of these clean drawings is $\dd$, and the other one, which we call $\dd'$, is obtained from $\dd$ simply by a Whitney switching on $\{a,b\}$; thus $\dd'$ can be obtained from $\dd$ simply by the relabellings $v_1\lar v_2, v_3 \lar v_4, v_5 \lar v_6, f_1 \lar f_3$, and $f_2 \lar f_4$.

\begin{center}
\begin{figure}[ht!]
\scalebox{0.75}{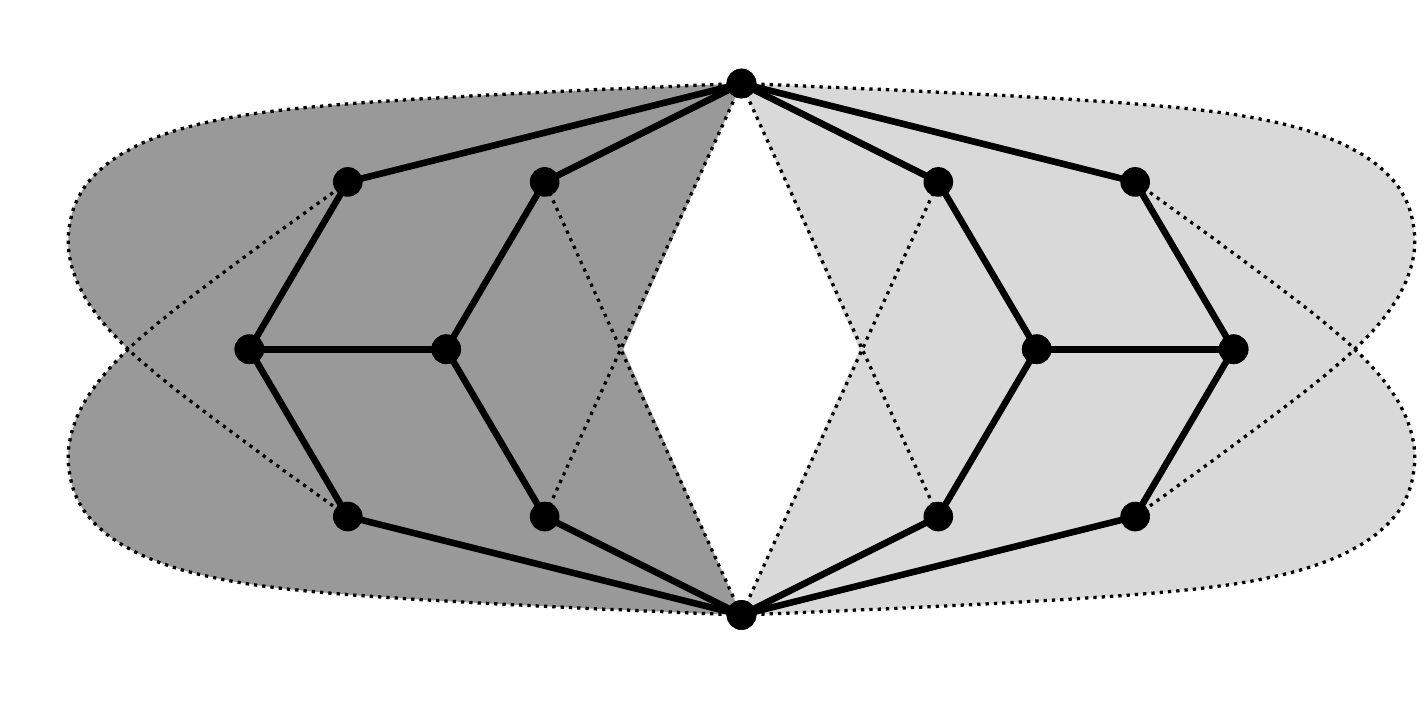}
\caption{The spherical drawing $\dd$.}
\label{fig:diam}
\end{figure}
\end{center}

Let $\dd_{\real^2}$ be a clean drawing of $G$ in $\real^2$. Clearly $\dd_{\real^2}$ can be obtained from a clean drawing of $G$ in $\sphe^2$ (that is, either $\dd$ or $\dd'$) by removing a point from a region (yielding the in\-fi\-nite region of $\dd_{\real^2}$), which we call the {\em special} region (of $\dd$ or $\dd'$). We suppose that $\dd_{\real^2}$ is obtained from $\dd$; a totally analogous argument is applied if $\dd_{\real^2}$ is obtained from $\dd'$.

We refer to the drawing $\dd$ in Figure~\ref{fig:diam}. If the special region is outside the darkly shaded area, then the restriction of $\dd_{\real^2}$ to the paths $u_5 a u_6$ and $u_1 b u_2$ is a clam; in this case $\dd_{\real^2}$ is not pseudolinear, by Proposition~\ref{pro:clam}. If the special region is  outside the lightly shaded area, then the restriction of $\dd_{\real^2}$ to the paths $v_5 a v_6$ and $v_1 b v_2$ is a clam; thus also in this case $\dd_{\real^2}$ is not pseudolinear, by Proposition~\ref{pro:clam}. We conclude that if $\dd_{\real^2}$ were pseudolinear, then the special region would have to be contained in {\em both} shaded areas. Since obviously no region satisfies this, we conclude that $\dd_{\real^2}$ is not pseudolinear.
\bqed \\

Let $G'$ be obtained by substituting each heavy edge by $m$ pairwise internally disjoint $2$-paths, and the edge $e_1$ by $k-3$ pairwise internally disjoint $2$-paths $P_1, P_2, \ldots, P_{k-3}$. By the Claim, in every pseudolinear drawing of $G$ some heavy edge is crossed. It follows that in every pseudolinear drawing of $G'$ at least $m$ edges are crossed, and so  $\pcr{G'} \ge m$. Since in a drawing of $G$ isomorphic to neither $\dd$ nor $\dd'$ some heavy edge is crossed, it follows that a drawing of $G'$ with fewer than $m$ crossings has $e_3$ crossing $e_4$, $f_3$ crossing $f_4$, $f_1$ crossing $f_2$, and $e_2$ crossing one edge of each path $P_i$, for $i=1,2,\ldots,k-3$. Thus such a drawing has at least $1+1 + 1 + (k-3) = k$ crossings, and so $\ucr{G'}\ge k$. Since a drawing of $G'$ with exactly $k$ crossings is  obtained from $\dd$ by drawing all the paths $P_i$ very close to $e_1$, we obtain $\ucr{G'} \le k$. Thus $\ucr{G'}=k$.
\end{proof}


\section{Separating $\trcr$ from $\tpcr$: proof of Theorems~\ref{thm:pseu-rec} and~\ref{thm:same}}\label{sec:pseu-rec}

To prove Theorems~\ref{thm:pseu-rec} and~\ref{thm:same} we proceed as suggested by Schaefer in~\cite{survey}. We make use of weighted graphs, whose definition and main properties are reviewed in Section~\ref{sub:wei}.  We start with a pseudoline arrangement $\aa$, and construct from $\aa$ a parameterized (by an integer $m\ge 1$) family of weighted graphs $(\ga,w_m)$; this is done in Section~\ref{sub:cons}. We then determine $\pcr{\ga,w_m}$, and bound by below $\rcr{\ga,w_m}$ (Section~\ref{sub:esti}). The key property (cf.~Propositions~\ref{pro:claima} and~\ref{pro:claimb}) is that $\rcr{\ga,w_m}$ is strictly greater than $\pcr{\ga,w_m}$ if and only if $\aa$ is non-stretchable.  Theorems~\ref{thm:pseu-rec} and~\ref{thm:same} then follow easily (Section~\ref{sub:proofs}).

\subsection{Weighted graphs and crossing numbers}\label{sub:wei}

We make essential use of weighted graphs, a simple device exploited in several crossing number constructions (see for instance~\cites{devos,dvorak}).

We recall that a {\em weighted} graph is a pair $(G,w)$, where $G$ is a graph and $w$ is a {\em weight function} $w:E(G) \to \natu$. A drawing of $(G,w)$ is simply any drawing of $G$, but the caveat is that in a drawing $\dd$ of $(G,w)$, a crossing between edges $e,f$ contributes $w(e)w(f)$ to the {\em weighted} crossing number $\ucr{\dd}$ of $\dd$. The {\em weighted crossing number} $\ucr{G,w}$ of $(G,w)$ is then the minimum $\ucr{\dd}$ over all drawings $\dd$ of $(G,w)$. (The weighted pseudolinear and rectilinear crossing numbers are analogously defined). Weighted graphs are a useful artifice for many crossing number related constructions, via the idea that $(G,w)$ can be turned into an ordinary, simple graph $G'$ by replacing each edge $e$ with a collection $\pp(e)$ of $w(e)$ internally disjoint $2$-paths with the same endpoints as $e$. We say that $G'$ is the simple graph {\em associated to} the weighted graph $(G,w)$. 

\begin{prop}\label{pro:canbe}
Let $(G,w)$ be a simple weighted graph, and let $G'$ be its associated simple graph. Then:
\begin{description}
\item{(a)} $\ucr{G,w} = \ucr{G'}$.
\item{(b)} $\pcr{G,w} = \pcr{G'}$.
\item{(c)} $\rcr{G,w} = \rcr{G'}$.
\end{description}
\end{prop}

\begin{proof}
Take a drawing $\dd$ in which $\ucr{G,w}$ is attained, and then, for each edge $e$ of $G$, draw the $w(e)$ $2$-paths in $\pp(e)$ sufficiently close to $e$ so that the following holds for all edges $e', e''$: a $2$-path of $\pp(e')$ crosses a $2$-path of $\pp(e'')$ if and only if $e'$ crosses $e''$ in $\dd$. This shows that $\ucr{G'} \le \ucr{G,w}$. For the reverse inequality, note that it is always possible to have a crossing-minimal drawing of $G'$ where the $2$-paths of $\pp(e)$ can be drawn sufficiently close to each other, so that a $2$-path in $\pp(e)$ crosses a $2$-path in $\pp(f)$ if and only if every $2$-path of $\pp(e)$ crosses every $2$-path of $\pp(f)$. It follows that we can regard the collection of $2$-paths $\pp(e)$ as a weighted edge. Thus $\ucr{G,w} \le \ucr{G'}$, 
and so (a) follows. For (b), we only need the additional observation that each collection $\pp(e)$ can be drawn so that each edge in $\pp(e)$ can be extended to a pseudoline, so that the final result is a pseudoline arrangement (see Figure~\ref{fig:howto}). The proof of (c) is totally analogous.
\end{proof}

\begin{figure}[ht!]
\centering
\scalebox{0.6}{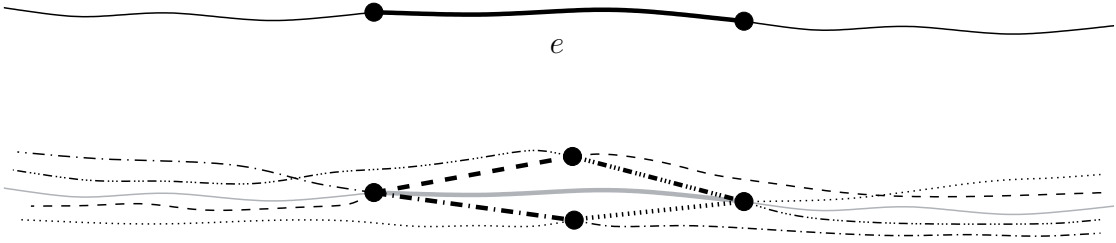}
\caption{\footnotesize Above we show an edge $e$ of weight $2$ in a pseudolinear drawing of a weighted graph $(G,w)$; the extension of $e$ to a pseudoline is also shown. Below we illustrate how to replace $e$ by $\pp(e)$ (two internally disjoint $2$-paths), and how to extend each of these $4$ edges to a pseudoline, so that the result is a pseudoline arrangement. By doing a similar operation on each edge of $(G,w)$, we obtain a pseudolinear drawing of a simple graph $G'$ such that $\pcr{G'}=\pcr{G,w}$.}
\label{fig:howto}
\end{figure}

\subsection{Construction of the graphs $(\ga,w_m)$}\label{sub:cons}

For each integer $m\ge 1$, we describe a construction of a weighted graph $(\ga,w_m)$, based on an (any) arrangement $\aa$ of pseudolines, presented as a {\em wiring diagram} (every arrangement of pseudolines can be so represented, as shown by Goodman~\cite{Go80}).  Let $s:=|\aa|$, and let $[s]=\{1,2,\ldots,s\}$. Suppose that the pseudolines of $\aa$ are labelled $\ell_1, \ell_2, \ldots, \ell_s$, according to the order in which they intersect a vertical line in the leftmost part of the wiring diagram (see Figure~\ref{fig:ringel} for the case in which $\aa$ is Ringel's non-stretchable arrangement of $9$ pseudolines).

\begin{figure}[ht!]
\centering
\scalebox{0.2}{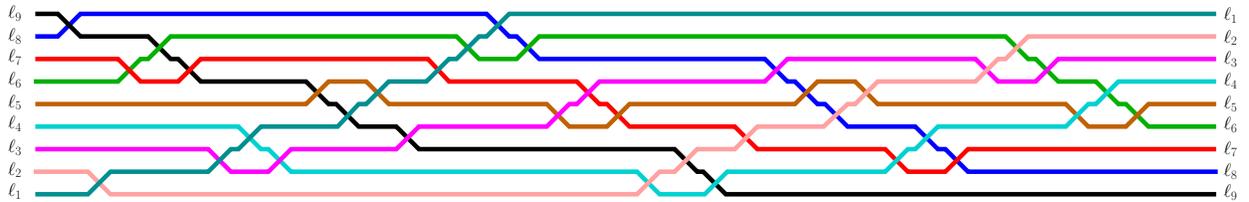}
\caption{Ringel's non-stretchable pseudoline arrangement $\rr$, as a wiring diagram.}
\label{fig:ringel}
\end{figure}

For each $i\in [s]$ add two copies of $\ell_i$, drawn very close to $\ell_i$: a pseudoline $\ell_i'$ slightly above $\ell_i$, and another pseudoline $\ell_i''$ slightly below $\ell_i$. Then transform this into (a drawing of) a graph by converting each of the $3s$ left-hand side endpoints and each of the $3s$ right hand-side endpoints into (degree $1$) vertices, and by transforming into a degree $4$ vertex each crossing of an $\ell_i'$ with an $\ell_j''$. (The remaining $5\binom{s}{2}$ crossings are not converted into vertices).

Before continuing with the construction, we label some of the current objects. For each $i\in [s]$: (i) label $a_i$ (respectively, $b_i$) the degree $1$ vertex on the left (respectively, right) hand side incident with $\ell_i$; (ii) label $u_i$ (respectively, $y_i$) the degree $1$ vertex on the left (respectively, right) hand side incident with $\ell_i'$; and (iii) label $v_i$ (respectively, $z_i$) the degree $1$ vertex on the left (respectively, right) hand side incident with $\ell_i''$. Thus for each $i\in [s]$, there is an edge $e_i$ joining $a_i$ to $b_i$ ($\ell_i$ is the arc representing $e_i$); there is a path $P_i$ joining $u_i$ to $y_i$ ($\ell_i'$ is the drawing of this path); and there is a path $Q_i$ joining $v_i$ to $z_i$ ($\ell_i''$ is the drawing of this path). 

Now add the necessary edges to obtain a cycle $C = v_1 a_1 u_1 v_2 a_2$  $u_2 \cdots v_s a_s u_s y_1 b_1 z_1 y_2 b_2 z_2 \cdots $ $\cdots y_s b_s z_s$. Finally, add two vertices $a,b$, and make $a$ adjacent to $a_i, u_i$, and $v_i$ for every $i\in[s]$, and make $b$ adjacent to $b_i, y_i$, and $z_i$ for every $i\in[s]$. Let $\ga$ denote the constructed graph. To help comprehension, we color {\em black} the edges that are either in $C$ or incident with $a$ or $b$; color {\em blue} the edges in $\cup_{i=1}^s P_i \cup Q_i$; and {\em red} the edges $e_1, e_2, \ldots, e_s$. In Figure~\ref{fig:showing} we illustrate how to turn an arrangement (wiring diagram) $\aa$ of $2$ pseudolines into the graph $\ga$.

Now for each positive integer $m$, we turn $\ga$ into a weighted graph $(\gam)$ as follows. Assign to each black edge a weight of $k:=\binom{s}{2}(1+4m)+2m$; assign to each blue edge a weight of $m$; and assign to each red edge a weight of $1$.

\begin{figure}[ht!]
\centering
\scalebox{0.95}{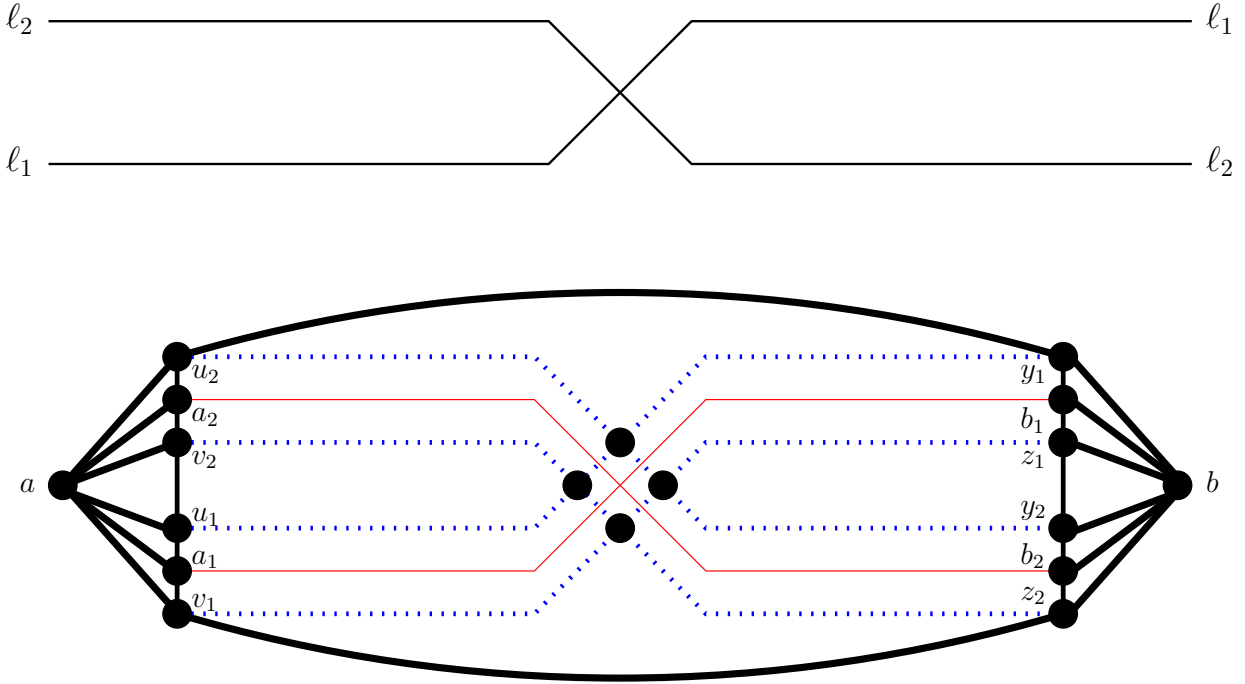}
\caption{Let $\aa$ be the arrangement with two pseudolines $\ell_1, \ell_2$ given above as a wiring diagram. Below we draw the graph $\ga$. The red edges $e_1$ and $e_2$ are drawn as thin, continuous arcs; the blue edges are dotted; and the black edges are thick.}
\label{fig:showing}
\end{figure}

\subsection{Determining $\pcr{\gam}$ and bounding $\rcr{\gam}$}\label{sub:esti}

First we determine $\pcr{\gam}$, and then we find a lower bound for $\rcr{\gam}$.

\begin{prop}\label{pro:claima}
$\pcr{\gam}=\binom{s}{2}(1+4m)$. If $\aa$ is stretchable, then
$\rcr{\gam}$ also equals $\binom{s}{2}(1+4m)$. 
\end{prop}

\begin{proof}
It is not difficult to verify that the drawing of $(\gam)$ described in the construction is pseudolinear. We claim that this drawing has exactly $\binom{s}{2}(1+4m)$ crossings. Indeed, for all $i,j\in [s]$, $i\neq j$, edges $e_i$ and $e_j$ cross each other, yielding $\binom{s}{2}$ crossings. Also, each red edge crosses $2(s-1)$ blue edges (for all $i,j\in [s], i\neq j$, the edge $e_i$ crosses both $P_j$ and $Q_j$). Since each blue-red crossing contributes $m$ to the crossing number, we have in total $\binom{s}{2} + s\cdot 2(s-1) \cdot m = \binom{s}{2}(1 + 4m)$ crossings. Thus  $\pcr{\gam} \le \binom{s}{2}(1+4m)$. 

Now let $\dd$ be a (not necessarily pseudolinear) crossing-minimal drawing of $(\gam)$. We note that since each black edge has weight greater than $\binom{s}{2}(1+4m)$, no black edge can be crossed in $\dd$. We may then assume without loss of generality that in $\dd$ the paths $P_i$ and $Q_i$, and the edges $e_i$, are all drawn inside the disk bounded by $C$.

Now for $i,j\in [s], i\neq j$, (i) the endpoints of $e_i$ and $e_j$ are alternating along $C$; (ii) the endpoints of $e_i$ and $P_j$ are alternating along $C$; and (iii) the endpoints of $e_i$ and $Q_j$ are alternating along $C$. Thus for all such $i,j$, $e_i$ crosses $e_j$, and $e_i$ crosses $P_j$ and also $Q_j$. Recalling again that blue-red crossings contribute $m$ to the crossing number, it follows that $\dd$ has at least  $\binom{s}{2} + s\cdot (s-1) \cdot 2m = \binom{s}{2}(1 + 4m)$ crossings. Thus  $\ucr{\gam}$ (and, consequently, $\pcr{\gam}$) is at least $\binom{s}{2}(1+4m)$. 

For the rectilinear crossing number part it suffices to prove that if $\aa$ is stretchable, then there is a rectilinear drawing of $(\gam)$ with exactly $\binom{s}{2}(1+4m)$ crossings. Suppose then that $\aa$ is stretchable. It is an easy exercise to show that then $e_1, e_2, \ldots, e_s$ can be drawn as straight lines in the plane so that each of them has one endpoint on the line $x=0$ and the other endpoint on the line $x=1$, so that the result is an arrangement isomorphic to $\aa$. It is then straightforward to add $P_i, Q_i, C, a,b$, and the edges incident with $a$ and $b$, so that every edge is a straight segment.
\end{proof}

\begin{prop}\label{pro:claimb}
If $\aa$ is non-stretchable, then $\rcr{\gam} \ge \binom{s}{2}(1+4m)+m$.
\end{prop}

\begin{proof}
Suppose that $\aa$ is non-stretchable. Let $\dd$ be a crossing-minimal rectilinear drawing of $(\gam)$. As in the proof of Proposition~\ref{pro:claima}, no black edge may be crossed in $\dd$, and we may assume without any loss of generality that all the paths $P_i, Q_i$, and all the edges $e_i$ are drawn inside the disk bounded by $C$. For each $i\in [s]$, the path $P_i$ cannot cross $Q_i$, as otherwise this would add at least $m^2$ crossings to the $\binom{s}{2}(1+4m)$ crossings already counted in the proof of Proosition~\ref{pro:claima}. On the other hand, if for every $i\in [s]$ the edge $e_i$ crosses neither $P_i$ or $Q_i$, then the drawing induced by $\cup_{i=1}^s P_i$ forms an arrangement isomorphic to $\aa$; the same conclusion holds for $\cup_{i=1}^s Q_i$. If no edge $e_i$ crosses $P_i\cup Q_i$, then every $e_i$ must be drawn inside the strip bounded by $P_i\cup Q_i$, and so it follows that the drawings of the edges $e_1, e_2, \ldots, e_s$ would form a straight line arrangement isomorphic to $\aa$, contradicting its nonstretchability. We conclude that for some $i\in [s]$, the edge $e_i$ must cross either $P_i$ or $Q_i$. In either case, the crossing contributes $m$ to $\rcr{\gam}$, in addition to the $\binom{s}{2}(1+4m)$ crossings already counted in the proof of Proposition~\ref{pro:claima}.
\end{proof}

\subsection{Proofs of Theorems~\ref{thm:pseu-rec} and~\ref{thm:same}}\label{sub:proofs}

\begin{proof}[Proof of Theorem~\ref{thm:pseu-rec}]
Let $\rr$ denote Ringel's non-stretch\-able arrangement with $9$ pseudolines. Theorem~\ref{thm:pseu-rec} follows at once using $(\gar)$, by combining Proposition~\ref{pro:canbe} (b) and (c) with Propositions~\ref{pro:claima} and~\ref{pro:claimb}.
\end{proof}

Let us denote {\sc PCN$\stackrel{?}{=}$RCN} the decision problem of determining if the pseudolinear crossing number and the rectilinear crossing number of an input graph are the same. Shor~\cite{Sh91} proved that {\sc Stretchability} (the problem of deciding if a pseudoline arrangement is stretchable) is NP-complete. By Mn\"ev's universality theorem~\cite{mnev}, it follows that {\sc Stretchability} is $\exists\real$-complete (cf.~\cite{Sch10}). We make a reduction to this problem to prove Theorem~\ref{thm:same}.  

\begin{proof}[Proof of Theorem~\ref{thm:same}]
We prove that {\sc Stretchability} $\propto$ {\sc PCN$\stackrel{?}{=}$RCN}. Let $\aa$ be a pseudoline arrangement, and consider the weighted graph $(\ga,w_1)$, which is clearly constructed from $\aa$ in polynomial time. Thus it suffices to prove that the answer to ``Is $\aa$ stretchable?'' is yes if and only if the answer to ``Is $\pcr{\ga,w_1} = \rcr{\ga,w_1}$?'' is yes. But this follows immediately from Propositions~\ref{pro:claima} and~\ref{pro:claimb}.
\end{proof}


\section{Concluding Remarks}\label{sec:cr}

In Theorem~\ref{thm:pseu-rec} we proved that there exist arbitrarily large graphs $G$ such that (roughly) $\rcr{G} \ge (145/144)\pcr{G}$. At the end of his survey~\cite{survey}, Schaefer asked if there is a function $f$ such that, for every graph $G$, $\rcr{G} \le f(\pcr{G})$. The existence (or not) of such an $f$ remains an important open question.

As Bienstock and Dean~\cite{BiDe93}, we make essential use of weighted graphs. Equivalently, we allow the existence of collections of internally disjoint $2$-paths with common endpoints; as a result we get simple (ordinary, unweighted) graphs, but these graphs are clearly not $3$-connected. Are these artifices really necessary to construct graphs with fixed crossing number and arbitrarily large rectilinear (or pseudolinar) crossing number? After unsuccessfully investigating this issue, we are willing to put forward the following.

\begin{conj}
There is a function $f$ such that for every $3$-connected graph $G$, $\rcr{G} \le f(\ucr{G})$.
\end{conj}





\section*{Acknowledgements}

We thank Marcus Schaefer for several insightful suggestions and remarks.


\end{document}

%% file: dis07.pdf_t
\begin{picture}(0,0)%
\includegraphics{dis07.pdf}%
\end{picture}%
\setlength{\unitlength}{4144sp}%
\begingroup\makeatletter\ifx\SetFigFont\undefined%
\gdef\SetFigFont#1#2#3#4#5{%
  \reset@font\fontsize{#1}{#2pt}%
  \fontfamily{#3}\fontseries{#4}\fontshape{#5}%
  \selectfont}%
\fi\endgroup%
\begin{picture}(5430,3970)(3091,-3527)
\put(4051,-196){\makebox(0,0)[lb]{\smash{{\SetFigFont{20}{24.0}{\familydefault}{\mddefault}{\updefault}{\color[rgb]{0,0,0}$e_1$}%
}}}}
\put(4051,-2491){\makebox(0,0)[lb]{\smash{{\SetFigFont{20}{24.0}{\familydefault}{\mddefault}{\updefault}{\color[rgb]{0,0,0}$e_2$}%
}}}}
\put(7246,-16){\makebox(0,0)[lb]{\smash{{\SetFigFont{20}{24.0}{\familydefault}{\mddefault}{\updefault}{\color[rgb]{0,0,0}$e_3$}%
}}}}
\put(7246,-2581){\makebox(0,0)[lb]{\smash{{\SetFigFont{20}{24.0}{\familydefault}{\mddefault}{\updefault}{\color[rgb]{0,0,0}$e_4$}%
}}}}
\put(3106,-1321){\makebox(0,0)[lb]{\smash{{\SetFigFont{20}{24.0}{\familydefault}{\mddefault}{\updefault}{\color[rgb]{0,0,0}$u$}%
}}}}
\put(8506,-1321){\makebox(0,0)[lb]{\smash{{\SetFigFont{20}{24.0}{\familydefault}{\mddefault}{\updefault}{\color[rgb]{0,0,0}$v$}%
}}}}
\put(5806,164){\makebox(0,0)[lb]{\smash{{\SetFigFont{20}{24.0}{\familydefault}{\mddefault}{\updefault}{\color[rgb]{0,0,0}$x$}%
}}}}
\put(5806,-2761){\makebox(0,0)[lb]{\smash{{\SetFigFont{20}{24.0}{\familydefault}{\mddefault}{\updefault}{\color[rgb]{0,0,0}$y$}%
}}}}
\put(4681,-871){\makebox(0,0)[lb]{\smash{{\SetFigFont{20}{24.0}{\familydefault}{\mddefault}{\updefault}{\color[rgb]{0,0,0}$w$}%
}}}}
\put(4681,-1681){\makebox(0,0)[lb]{\smash{{\SetFigFont{20}{24.0}{\familydefault}{\mddefault}{\updefault}{\color[rgb]{0,0,0}$z$}%
}}}}
\put(6976,-1726){\makebox(0,0)[lb]{\smash{{\SetFigFont{20}{24.0}{\familydefault}{\mddefault}{\updefault}{\color[rgb]{0,0,0}$t$}%
}}}}
\put(6976,-871){\makebox(0,0)[lb]{\smash{{\SetFigFont{20}{24.0}{\familydefault}{\mddefault}{\updefault}{\color[rgb]{0,0,0}$s$}%
}}}}
\end{picture}%

%% file: dis05.pdf_t
\begin{picture}(0,0)%
\includegraphics{dis05.pdf}%
\end{picture}%
\setlength{\unitlength}{4144sp}%
\begingroup\makeatletter\ifx\SetFigFont\undefined%
\gdef\SetFigFont#1#2#3#4#5{%
  \reset@font\fontsize{#1}{#2pt}%
  \fontfamily{#3}\fontseries{#4}\fontshape{#5}%
  \selectfont}%
\fi\endgroup%
\begin{picture}(6510,3247)(2011,-11213)
\put(5356,-8161){\makebox(0,0)[lb]{\smash{{\SetFigFont{14}{16.8}{\familydefault}{\mddefault}{\updefault}{\color[rgb]{0,0,0}$a$}%
}}}}
\put(5356,-11131){\makebox(0,0)[lb]{\smash{{\SetFigFont{14}{16.8}{\familydefault}{\mddefault}{\updefault}{\color[rgb]{0,0,0}$b$}%
}}}}
\put(4186,-8836){\makebox(0,0)[lb]{\smash{{\SetFigFont{14}{16.8}{\familydefault}{\mddefault}{\updefault}{\color[rgb]{0,0,0}$u_2$}%
}}}}
\put(4186,-9601){\makebox(0,0)[lb]{\smash{{\SetFigFont{14}{16.8}{\familydefault}{\mddefault}{\updefault}{\color[rgb]{0,0,0}$u_4$}%
}}}}
\put(8506,-9151){\makebox(0,0)[lb]{\smash{{\SetFigFont{14}{16.8}{\familydefault}{\mddefault}{\updefault}{\color[rgb]{0,0,0}$f_1$}%
}}}}
\put(8506,-10186){\makebox(0,0)[lb]{\smash{{\SetFigFont{14}{16.8}{\familydefault}{\mddefault}{\updefault}{\color[rgb]{0,0,0}$f_2$}%
}}}}
\put(2026,-9106){\makebox(0,0)[lb]{\smash{{\SetFigFont{14}{16.8}{\familydefault}{\mddefault}{\updefault}{\color[rgb]{0,0,0}$e_1$}%
}}}}
\put(2026,-10186){\makebox(0,0)[lb]{\smash{{\SetFigFont{14}{16.8}{\familydefault}{\mddefault}{\updefault}{\color[rgb]{0,0,0}$e_2$}%
}}}}
\put(3691,-10231){\makebox(0,0)[lb]{\smash{{\SetFigFont{14}{16.8}{\familydefault}{\mddefault}{\updefault}{\color[rgb]{0,0,0}$u_5$}%
}}}}
\put(4636,-10321){\makebox(0,0)[lb]{\smash{{\SetFigFont{14}{16.8}{\familydefault}{\mddefault}{\updefault}{\color[rgb]{0,0,0}$u_6$}%
}}}}
\put(2836,-9601){\makebox(0,0)[lb]{\smash{{\SetFigFont{14}{16.8}{\familydefault}{\mddefault}{\updefault}{\color[rgb]{0,0,0}$u_3$}%
}}}}
\put(3286,-8701){\makebox(0,0)[lb]{\smash{{\SetFigFont{14}{16.8}{\familydefault}{\mddefault}{\updefault}{\color[rgb]{0,0,0}$u_1$}%
}}}}
\put(5581,-9916){\makebox(0,0)[lb]{\smash{{\SetFigFont{14}{16.8}{\familydefault}{\mddefault}{\updefault}{\color[rgb]{0,0,0}$f_4$}%
}}}}
\put(5086,-9106){\makebox(0,0)[lb]{\smash{{\SetFigFont{14}{16.8}{\familydefault}{\mddefault}{\updefault}{\color[rgb]{0,0,0}$e_3$}%
}}}}
\put(5041,-9916){\makebox(0,0)[lb]{\smash{{\SetFigFont{14}{16.8}{\familydefault}{\mddefault}{\updefault}{\color[rgb]{0,0,0}$e_4$}%
}}}}
\put(5491,-9106){\makebox(0,0)[lb]{\smash{{\SetFigFont{14}{16.8}{\familydefault}{\mddefault}{\updefault}{\color[rgb]{0,0,0}$f_3$}%
}}}}
\put(6931,-10231){\makebox(0,0)[lb]{\smash{{\SetFigFont{14}{16.8}{\familydefault}{\mddefault}{\updefault}{\color[rgb]{0,0,0}$v_5$}%
}}}}
\put(6391,-8836){\makebox(0,0)[lb]{\smash{{\SetFigFont{14}{16.8}{\familydefault}{\mddefault}{\updefault}{\color[rgb]{0,0,0}$v_2$}%
}}}}
\put(7201,-8656){\makebox(0,0)[lb]{\smash{{\SetFigFont{14}{16.8}{\familydefault}{\mddefault}{\updefault}{\color[rgb]{0,0,0}$v_1$}%
}}}}
\put(6436,-9601){\makebox(0,0)[lb]{\smash{{\SetFigFont{14}{16.8}{\familydefault}{\mddefault}{\updefault}{\color[rgb]{0,0,0}$v_4$}%
}}}}
\put(7741,-9601){\makebox(0,0)[lb]{\smash{{\SetFigFont{14}{16.8}{\familydefault}{\mddefault}{\updefault}{\color[rgb]{0,0,0}$v_3$}%
}}}}
\put(5986,-10276){\makebox(0,0)[lb]{\smash{{\SetFigFont{14}{16.8}{\familydefault}{\mddefault}{\updefault}{\color[rgb]{0,0,0}$v_6$}%
}}}}
\end{picture}%

%% file: y02.pdf_t
\begin{picture}(0,0)%
\includegraphics{y02.pdf}%
\end{picture}%
\setlength{\unitlength}{4144sp}%
\begingroup\makeatletter\ifx\SetFigFont\undefined%
\gdef\SetFigFont#1#2#3#4#5{%
  \reset@font\fontsize{#1}{#2pt}%
  \fontfamily{#3}\fontseries{#4}\fontshape{#5}%
  \selectfont}%
\fi\endgroup%
\begin{picture}(11114,2354)(-561,-4343)
\put(4906,-2491){\makebox(0,0)[lb]{\smash{{\SetFigFont{20}{24.0}{\familydefault}{\mddefault}{\updefault}{\color[rgb]{0,0,0}$e$}%
}}}}
\end{picture}%

%% file: wi01.pdf_t
\begin{picture}(0,0)%
\includegraphics{wi01.pdf}%
\end{picture}%
\setlength{\unitlength}{4144sp}%
\begingroup\makeatletter\ifx\SetFigFont\undefined%
\gdef\SetFigFont#1#2#3#4#5{%
  \reset@font\fontsize{#1}{#2pt}%
  \fontfamily{#3}\fontseries{#4}\fontshape{#5}%
  \selectfont}%
\fi\endgroup%
\begin{picture}(36390,5995)(751,-6632)
\put(37126,-3796){\makebox(0,0)[lb]{\smash{{\SetFigFont{34}{40.8}{\familydefault}{\mddefault}{\updefault}{\color[rgb]{0,0,0}$\ell_5$}%
}}}}
\put(37126,-3121){\makebox(0,0)[lb]{\smash{{\SetFigFont{34}{40.8}{\familydefault}{\mddefault}{\updefault}{\color[rgb]{0,0,0}$\ell_4$}%
}}}}
\put(37126,-2446){\makebox(0,0)[lb]{\smash{{\SetFigFont{34}{40.8}{\familydefault}{\mddefault}{\updefault}{\color[rgb]{0,0,0}$\ell_3$}%
}}}}
\put(37126,-1816){\makebox(0,0)[lb]{\smash{{\SetFigFont{34}{40.8}{\familydefault}{\mddefault}{\updefault}{\color[rgb]{0,0,0}$\ell_2$}%
}}}}
\put(37126,-1141){\makebox(0,0)[lb]{\smash{{\SetFigFont{34}{40.8}{\familydefault}{\mddefault}{\updefault}{\color[rgb]{0,0,0}$\ell_1$}%
}}}}
\put(37126,-4426){\makebox(0,0)[lb]{\smash{{\SetFigFont{34}{40.8}{\familydefault}{\mddefault}{\updefault}{\color[rgb]{0,0,0}$\ell_6$}%
}}}}
\put(37126,-5146){\makebox(0,0)[lb]{\smash{{\SetFigFont{34}{40.8}{\familydefault}{\mddefault}{\updefault}{\color[rgb]{0,0,0}$\ell_7$}%
}}}}
\put(37126,-5821){\makebox(0,0)[lb]{\smash{{\SetFigFont{34}{40.8}{\familydefault}{\mddefault}{\updefault}{\color[rgb]{0,0,0}$\ell_8$}%
}}}}
\put(37126,-6451){\makebox(0,0)[lb]{\smash{{\SetFigFont{34}{40.8}{\familydefault}{\mddefault}{\updefault}{\color[rgb]{0,0,0}$\ell_9$}%
}}}}
\put(766,-1096){\makebox(0,0)[lb]{\smash{{\SetFigFont{34}{40.8}{\familydefault}{\mddefault}{\updefault}{\color[rgb]{0,0,0}$\ell_9$}%
}}}}
\put(766,-1771){\makebox(0,0)[lb]{\smash{{\SetFigFont{34}{40.8}{\familydefault}{\mddefault}{\updefault}{\color[rgb]{0,0,0}$\ell_8$}%
}}}}
\put(766,-2401){\makebox(0,0)[lb]{\smash{{\SetFigFont{34}{40.8}{\familydefault}{\mddefault}{\updefault}{\color[rgb]{0,0,0}$\ell_7$}%
}}}}
\put(766,-3751){\makebox(0,0)[lb]{\smash{{\SetFigFont{34}{40.8}{\familydefault}{\mddefault}{\updefault}{\color[rgb]{0,0,0}$\ell_5$}%
}}}}
\put(766,-4381){\makebox(0,0)[lb]{\smash{{\SetFigFont{34}{40.8}{\familydefault}{\mddefault}{\updefault}{\color[rgb]{0,0,0}$\ell_4$}%
}}}}
\put(766,-5101){\makebox(0,0)[lb]{\smash{{\SetFigFont{34}{40.8}{\familydefault}{\mddefault}{\updefault}{\color[rgb]{0,0,0}$\ell_3$}%
}}}}
\put(766,-5776){\makebox(0,0)[lb]{\smash{{\SetFigFont{34}{40.8}{\familydefault}{\mddefault}{\updefault}{\color[rgb]{0,0,0}$\ell_2$}%
}}}}
\put(766,-6406){\makebox(0,0)[lb]{\smash{{\SetFigFont{34}{40.8}{\familydefault}{\mddefault}{\updefault}{\color[rgb]{0,0,0}$\ell_1$}%
}}}}
\put(766,-3076){\makebox(0,0)[lb]{\smash{{\SetFigFont{34}{40.8}{\familydefault}{\mddefault}{\updefault}{\color[rgb]{0,0,0}$\ell_6$}%
}}}}
\end{picture}%

%% file: z02.pdf_t
\begin{picture}(0,0)%
\includegraphics{z02.pdf}%
\end{picture}%
\setlength{\unitlength}{4144sp}%
\begingroup\makeatletter\ifx\SetFigFont\undefined%
\gdef\SetFigFont#1#2#3#4#5{%
  \reset@font\fontsize{#1}{#2pt}%
  \fontfamily{#3}\fontseries{#4}\fontshape{#5}%
  \selectfont}%
\fi\endgroup%
\begin{picture}(7590,4357)(-644,-5357)
\put(5761,-3976){\makebox(0,0)[lb]{\smash{{\SetFigFont{12}{14.4}{\familydefault}{\mddefault}{\updefault}{\color[rgb]{0,0,0}$z_1$}%
}}}}
\put(5761,-3436){\makebox(0,0)[lb]{\smash{{\SetFigFont{12}{14.4}{\familydefault}{\mddefault}{\updefault}{\color[rgb]{0,0,0}$y_1$}%
}}}}
\put(5761,-4831){\makebox(0,0)[lb]{\smash{{\SetFigFont{12}{14.4}{\familydefault}{\mddefault}{\updefault}{\color[rgb]{0,0,0}$z_2$}%
}}}}
\put(5761,-3751){\makebox(0,0)[lb]{\smash{{\SetFigFont{12}{14.4}{\familydefault}{\mddefault}{\updefault}{\color[rgb]{0,0,0}$b_1$}%
}}}}
\put(5761,-4606){\makebox(0,0)[lb]{\smash{{\SetFigFont{12}{14.4}{\familydefault}{\mddefault}{\updefault}{\color[rgb]{0,0,0}$b_2$}%
}}}}
\put(5761,-4291){\makebox(0,0)[lb]{\smash{{\SetFigFont{12}{14.4}{\familydefault}{\mddefault}{\updefault}{\color[rgb]{0,0,0}$y_2$}%
}}}}
\put(541,-3436){\makebox(0,0)[lb]{\smash{{\SetFigFont{12}{14.4}{\familydefault}{\mddefault}{\updefault}{\color[rgb]{0,0,0}$u_2$}%
}}}}
\put(541,-3706){\makebox(0,0)[lb]{\smash{{\SetFigFont{12}{14.4}{\familydefault}{\mddefault}{\updefault}{\color[rgb]{0,0,0}$a_2$}%
}}}}
\put(541,-3976){\makebox(0,0)[lb]{\smash{{\SetFigFont{12}{14.4}{\familydefault}{\mddefault}{\updefault}{\color[rgb]{0,0,0}$v_2$}%
}}}}
\put(541,-4606){\makebox(0,0)[lb]{\smash{{\SetFigFont{12}{14.4}{\familydefault}{\mddefault}{\updefault}{\color[rgb]{0,0,0}$a_1$}%
}}}}
\put(-539,-4156){\makebox(0,0)[lb]{\smash{{\SetFigFont{12}{14.4}{\familydefault}{\mddefault}{\updefault}{\color[rgb]{0,0,0}$a$}%
}}}}
\put(6931,-4156){\makebox(0,0)[lb]{\smash{{\SetFigFont{12}{14.4}{\familydefault}{\mddefault}{\updefault}{\color[rgb]{0,0,0}$b$}%
}}}}
\put(-629,-2131){\makebox(0,0)[lb]{\smash{{\SetFigFont{14}{16.8}{\familydefault}{\mddefault}{\updefault}{\color[rgb]{0,0,0}$\ell_1$}%
}}}}
\put(-629,-1231){\makebox(0,0)[lb]{\smash{{\SetFigFont{14}{16.8}{\familydefault}{\mddefault}{\updefault}{\color[rgb]{0,0,0}$\ell_2$}%
}}}}
\put(6931,-1231){\makebox(0,0)[lb]{\smash{{\SetFigFont{14}{16.8}{\familydefault}{\mddefault}{\updefault}{\color[rgb]{0,0,0}$\ell_1$}%
}}}}
\put(6931,-2131){\makebox(0,0)[lb]{\smash{{\SetFigFont{14}{16.8}{\familydefault}{\mddefault}{\updefault}{\color[rgb]{0,0,0}$\ell_2$}%
}}}}
\put(541,-4336){\makebox(0,0)[lb]{\smash{{\SetFigFont{12}{14.4}{\familydefault}{\mddefault}{\updefault}{\color[rgb]{0,0,0}$u_1$}%
}}}}
\put(541,-4876){\makebox(0,0)[lb]{\smash{{\SetFigFont{12}{14.4}{\familydefault}{\mddefault}{\updefault}{\color[rgb]{0,0,0}$v_1$}%
}}}}
\end{picture}%